\UseRawInputEncoding
\documentclass[10pt]{article}
\usepackage{amsfonts}
\usepackage{amsmath,amssymb,amsthm,amsfonts,enumerate}
\usepackage{graphicx}
\usepackage{float}
\usepackage[scriptsize,nooneline]{subfigure}
\usepackage{indentfirst}
\usepackage{cite}
\usepackage{color}
\usepackage{mathrsfs}
\usepackage{epstopdf}
\usepackage{setspace}
\usepackage[labelfont=footnotesize,font=footnotesize]{caption}
\usepackage[colorlinks, citecolor=red]{hyperref}
\usepackage{setspace}

\newtheorem{theorem}{Theorem}[section]

\newtheorem{lemma}{Lemma}[section]

\newtheorem{definition}{Definition}[section]

\newtheorem{proposition}{Proposition}[section]
\newtheorem{remark}{Remark}[section]

\usepackage{hyperref}
\def\geq{\geqslant}\def\leq{\leqslant}
\def\ge{\geqslant}\def\le{\leqslant}

\def\lim{\limits}
\def\lt{\left}
\def\rt{\right}
\usepackage{indentfirst}

\begin{document}
\title{\bf\Large Boundedness of some operators on weighted amalgam spaces
\footnotetext{\hspace{-0.35cm}
\endgraf This project is supported
by the National Natural Science Foundation of China (Grant No.12061069)
.}}
\author{Yuan Lu$^1$, Songbai Wang$^2$, Jiang Zhou$^{1,}$\,\footnote{Corresponding author
e-mail: \texttt{Zhoujiang@xju.edu.cn}
}
\\[.5cm]
\small $^{1}$College of Mathematics and System Sciences, Xinjiang University, Urumqi 830046\\
\small People's Republic of China\\
\small  $^2$ College of Mathematics and Statistics, Chongqing Three Gorges University, Chongqing 404130\\
\small People's Republic of China}
\date{}
\maketitle

\vspace{-0.7cm}

\begin{center}

\begin{minipage}{13cm}
{
\small {\bf Abstract}\quad
Let $t\in(0,\infty)$, $p\in(1,\infty)$, $q\in[1,\infty]$, $w\in A_p$ and $v\in A_q$. We introduce the weighted amalgam space $(L^p,L^q)_t(\mathbb R^n)$
and show some properties of it.
Some estimates on these spaces for the classical operators in harmonic analysis, such as the Hardy--Littlewood maximal operator,
the Calder\'on--Zygmund operator, the Riesz potential, singular integral operators with the rough kernel,
the Marcinkiewicz integral, the Bochner-Riesz operator, the Littlewood-Paley $g$ function and the intrinsic square function, are considered.
Our main method is extrapolation. We obtain some new weak results for these operators on weighted amalgam spaces.
\\
}
{
{\bf MSC(2000) subject classification:}  42B25; 42B20.
}
\end{minipage}
\end{center}

\section{Introduction}\label{s1}

In 1926, Wiener\cite{1} first introduced amalgam spaces to formulate his generalized harmonic analysis.
For $p,\,q\in(0,\infty)$, the amalgam space $(L^{p},L^{q})(\mathbb R)$ is defined by
$$
(L^{p},L^{q})(\mathbb R):
=\lt\{f\in L_{\mathrm{loc}}^p:\,\lt[\sum_{n\in\mathbb Z}\lt(\int_n^{n+1}|f(x)|^p\,dx\rt)^\frac{q}{p}\rt]^\frac1q<\infty\rt\}.
$$

The first systematic study of these spaces was undertaken by Holland \cite{2} in 1975.
A large number of authors\cite{refJCC,4,5,6,7} research amalgam spaces or some applications of these spaces.
A significant difference in considering amalgam spaces instead of $L^ p$ spaces is that amalgam
spaces give information about the local $L^p$ , and global $L^q$, properties of the functions, while
$L^p$ spaces do not make that distinction.
To study the weak solutions of boundary value problems for a $t$-independent elliptic systems in the upper half plane,
Auscher and Mourgoglou \cite{AM} introduced the slice spaces in 2019. Moreover,  Auscher and Prisuelos-Arribas \cite{APA} studied the boundedness of
Calder\'on--Zygmund operators, the Hardy-Littlewood maximal operator and the fractional integral operator on slice spaces.
Recently, Ho \cite{Ho} obtained the boundedness of some classical operators, such as singular integral operators, the Fourier integral operator, the geometric
maximal operator, the maximal Bochner--Riesz means, the parametric Marcinkiewicz integral and the multiplier operators on Hardy Orlicz-slice spaces
introduced in \cite{ZYYW}.

The main purpose of this paper is to consider the boundedness of some operators on a new weighted amalgam space. To state our main results,
we recall some necessary definitions.
\par A weight $\omega$ is a positive and locally integrable function on $\mathbb R^n$.
For $p\in(0,\infty)$, the weighted Lebesgue spaces $L_\omega^p(\mathbb R^n)$
is defined as the set of all measurable functions $f$ on $\mathbb R^n$ such that
$$
\|f\|_{L_\omega^p(\mathbb R^n)}:=\lt[\int_{\mathbb R^n}|f(x)|^p\omega(x)\,dx\rt]^\frac1p<\infty.
$$
The weak weighted Lebesgue space $L^{p,\infty}_\omega(\mathbb R^n)$ is defined as the set of all measurable functions $f$ on $\mathbb R^n$ such that
$$
\|f\|_{L_\omega^{p,\infty}(\mathbb R^n)}:=\sup_{\alpha>0}\alpha\omega(\{x\in\mathbb R^n:\,|f(x)|>\alpha\})^\frac1p<\infty.
$$
For $p=\infty$,
$$
\|f\|_{L_\omega^{\infty}(\mathbb R^n)}:=\underset{x\in\mathbb R^n}{\mathrm{ess~sup}}|f(x)|<\infty.
$$
\begin{definition}\label{Defama}
Let $w,\,v$  be weights, $0<t<\infty$, $1<p<\infty$, and $1\leq q\leq \infty$.
We define the weighted amalgam space $(L_{w}^{p},L_{v}^{q})_t:= (L_{w}^{p},L_{v}^{q})_t(\mathbb R^n)$ as the
space of all measurable functions $f$ on $\mathbb R^n$ satisfying ${||f||}_{(L_{w}^{p},L_{v}^{q})_t}<\infty $,
where
$$
||f||_{(L_{w}^{p},L_{v}^{q})_t(\mathbb R^n)}:=\lt\|\lt(\frac{1}{w(B(\cdot,t))}\int_{B(\cdot,t)}
|f(y)|^pw(y)\,dy\rt)^\frac{1}{p}\rt\|_{L^q_v(\mathbb R^n)}.
$$
with the usual modification when $q=\infty$.
\end{definition}

\begin{remark}
If $w=v=1$, then $(L_{w}^{p},L_{v}^{q})(\mathbb R^n)_t$ is the slice space $(E_p^q)_t(\mathbb R^n)$
(see \cite{AM,APA}).
\end{remark}
\begin{definition}\label{DefWama}
Let $w,\,v$  be weights, $0<t<\infty$, $1<p<\infty$, and $1\leq q\leq \infty$.
The weak weighted amalgam space $W(L_{w}^{p},L^q_v)_t:= W(L^p_w,L^q_v)_t(\mathbb R^n)$
is defined as the space of all measurable functions $f$ on $\mathbb R^n$ satisfying
$||f||_{W(L_{w}^{p},L_{v}^q)_t}<\infty$,
where
$$
\lt\|f\rt\|_{W(L^p_w,L^q_v)_t(\mathbb R^n)}:=\underset{{\lambda>0}}{\mathop{\sup}}
\lambda\lt\|\chi_{\{x\in\mathbb R^n:|f(x)|>\lambda\}}\rt\|_{(L^p_w,L^q_v)_t(\mathbb R^n)}<\infty.
$$
\end{definition}
We still recall the definition of Muckenhoupt's weights $A_{p}(1\leq p\leq\infty)$.
These weights introduced in\cite{15} were used to characterize the boundedness of
the Hardy-Littlewood maximal operator on weighted Lebesgue spaces.
For a locllay integrable function $f$, we define the centered Hardy--Littlewood maximal operator,
for almost every $x\in\mathbb R^n$,
$$
Mf(x):=\sup_{\tau>0}\frac1{|B(x,\tau)|}\int_{B(x,\tau)}|f(y)|\,dy.
$$
\begin{definition}
Let $1<p<\infty$. A weight $w$ is said to be of class $A_p$ if
$$
\sup_{B\subset\mathbb R^n}\lt(\frac{1}{|B|}\int_{B}w(x)\,dx\rt)\lt(\frac{1}{|B|}
\int_{B}w(x)^{\frac{1}{1-p}} \,dx\rt)^{p-1}<\infty.
$$
A weight $w$ is said to be of class $A_1$ if
$$
M(w)(x)\leq Cw(x) \quad\quad for~almost~all~x\in\mathbb R^n
$$
for some positive constant $C$. We define $A_\infty:=\cup_{p\geq1}A_p$.
\end{definition}
Our main results are given as follows.
\begin{theorem}\label{ThMf}
Let $0<t<\infty$, $1<p<\infty$ and $w\in {{A}_{p}}$.\\
(a). If $1<q\leq\infty$ and $v\in A_q$, then we have
$$
||M(f)||_{(L_w^p,L_v^q)_t(\mathbb R^n)}\leq C ||f||_{(L_w^p,L_v^q)_t(\mathbb R^n)}.
$$
(b). If $q=1$ and $v\in A_1$, then we have
$$
||M(f)||_{W(L_w^p,L_v^1)_t(\mathbb R^n)}\leq  C||f||_{(L_w^p,L_v^1)_t(\mathbb R^n)}.
$$
The universal positive constant $C$ is independent of $f$ and $t$.
\end{theorem}
Let $\delta>0$ and let $\Delta$ be the off-diagonal in $\mathbb R^n\times\mathbb R^n$,
that is, $\Delta:=\mathbb R^n\times\mathbb R^n\setminus\{(x,x):\ x\in\mathbb R^n\}$.
The \emph{Calder\'on--Zygmund singular integral operator of non-convolution type} is a bounded linear operator
$T:\ L^2(\mathbb R^n)\rightarrow L^2(\mathbb R^n)$
satisfying that, for all $f\in C_c^\infty(\mathbb R^n)$ and $x\notin\mathrm{supp}\,(f)$,
$$
T(f)(x):=\int_{\mathbb R^n} K(x,y)f(y)dy,
$$
where the distributional kernel coincides with a locally integrable function $K$ defined away from the
diagonal on $\mathbb R^n\times\mathbb R^n$. When $K$ also satisfies that, for $x,\,y\in\mathbb R^n$ with $x\neq y$,
\begin{equation}\label{Eqsc}
|K(x,y)|\leq\frac{C_0}{|x-y|^n},
\end{equation}
\begin{equation}\label{Eqrc}
|K(x,y)-K(x,y+h)|+|K(x,y)-K(x+h,y)|\leq\frac{C_1|h|^\delta}{|x-y|^{n+\delta}},
\end{equation}
whenever $|x-y|\geq2|h|$, and we call $K$ the \emph{standard kernel}.

\begin{theorem}\label{ThTf}
Let $T$ be a Calder\'on-Zygmund operator with kernel $K$ satisfying \eqref{Eqsc} and \eqref{Eqrc}.
Suppose that $0<t<\infty$, $1<p<\infty$ and $w\in {{A}_{p}}$.\\
(a). If $1<q<\infty$ and $v\in A_q$, then we have
$$
||T(f)||_{(L_w^p,L_v^q)_t(\mathbb R^n)}\leq C ||f||_{(L_w^p,L_v^q)_t(\mathbb R^n)}.
$$
(b). If $q=1$ and $v\in A_1$, then we have
$$
||T(f)||_{W(L_w^p,L_v^1)_t(\mathbb R^n)}\leq  C||f||_{(L_w^p,L_v^1)_t(\mathbb R^n)}.
$$
The universal positive constant $C$ is independent of $f$ and $t$.
\end{theorem}

We also recall the definition of  $A_{p,q}$ weights which are closely related to the weighted boundedness of the fractional integral.
\begin{definition}
We say that a weight $w$ is said to be of class $A_{p,q}$, for $1<p,q<\infty$,
if there exists a constant $C > 0$ such that
$$
\sup_{B\subset\mathbb R^n}\lt(\frac{1}{|B|}\int_{B}w(x)^q\,dx\rt)^\frac{1}{q}\lt(\frac{1}{|B|}
\int_{B}w(x)^{-p'} dx\rt)^{\frac{1}{p'}}\leq C<\infty,
$$
where $p'$ is the conjugate exponent of $p$, that is, $\frac1p+\frac1{p'}=1$.
\par And say $w$ is in $A_{1,q}$ with $1<q<\infty$, if there exist constant $C>0$ such that
$$
\sup_{B\subset\mathbb R^n}\lt(\frac{1}{|B|}\int_{B}w(x)^q\,dx\rt)^\frac1q\lt(
\underset{B}{\mathrm{ess~sup}}\frac1{w(x)}\rt)\leq C<\infty.
$$
\end{definition}

\begin{proposition}\cite{refgf}
Suppose that $0<\alpha<n$, $1<p<\frac n\alpha$ and $\frac1q=\frac1p-\frac\alpha n$.
\begin{enumerate}[(i).]
         \item If $1<p$, then $w\in A_{p,q}$ if and only if $w^q\in A_{p\frac{n-\alpha}n}$;
         \item If $1<p$, $w\in A_{p,q}$, then $w^q\in A_q$ and $w^p\in A_p$;
         \item If $p=1$, then $w\in A_{1,q}$ if and only if $w^q\in A_1$.
\end{enumerate}
\end{proposition}
\begin{proposition}\label{Eqweight}\cite{CMFA}
Let $1\leq p<\infty$, $w\in A_p$, then
\begin{enumerate}[(i).]
  \item A weight $\omega$ is doubling, that is, for any ball $B\in\mathbb R^n$, $\omega(2B)\leq C\omega(B)$,
where the positive constant $C$ is independent of $B$;
  \item $\omega\in A_\infty$, for every ball $B$ and every measurable
set $E\subset B$, there exist $C, \delta>0$ such that
\begin{equation}\label{Eqw}
\frac{\omega(E)}{\omega(B)}\leq C\lt(\frac{|E|}{|B|}\rt)^\delta.
\end{equation}
\end{enumerate}
\end{proposition}
For $\alpha\in(0,n)$, the Riesz potential $I_\alpha$ is defined as follows.
\begin{definition}
Let $0<\alpha<n$, the Riesz potential $I_\alpha$ is defined by
$$
I_{\alpha }f(x):=\frac{1}{\gamma (\alpha )}\int_{\mathbb R^n}\frac{f(\xi)}{|x-\xi|^{n-\alpha}}\,d\xi,
$$
where $\gamma (\alpha )=\pi^{\frac{n}{2}}2^\alpha\Gamma(\frac{\alpha}{2})/\Gamma(\frac{n-\alpha}{2})$.
\end{definition}

\begin{theorem}\label{ThIf}
Let $0<\alpha<n$ and $0<t<\infty$. Suppose that $1<p,q<\frac n\alpha$ such that
$$
\frac{\alpha}{n}=\frac1{p_0}-\frac1p=\frac1{q_0}-\frac1q.
$$
Suppose that $w\in A_{p_0,p}$. \\
(a). If $1<q_0<\infty$, and $v\in A_{q_0,q}$, then we have
$$
||I_\alpha(f)||_{(L^p_w,L^q_v)_t(\mathbb R^n)}\leq C ||f||_{(L_{w}^{p_0},L_{v}^{q_0})_t(\mathbb R^n)}.
$$
(b). If $q_0=1$, and $v\in A_{1,q}$, then we have
$$
||I_\alpha(f)||_{W(L_w^p,L_{v}^q)_t(\mathbb R^n)}\leq  C||f||_{(L_w^{p_0},L_v^1)_t(\mathbb R^n)}.
$$
The universal positive constant $C$ is independent of $f$ and $t$.
\end{theorem}


We end this section by explaining some notations. Given a weight $w$ and a measurable set $B$, it can be denoted by
$w(B):=\int_Bw(x)dx$.
For $\alpha>0$ and a ball $B\subset\mathbb R^n$, $\alpha B$ is the ball with same center as $B$
and radius $\alpha$ times radius of $B$. We denote by $B^c:=\mathbb R^n\backslash B$
the complement of $B$.
We write $A\lesssim B$ to mean that there exists a positive constant $C$ such that $A\leq CB$.
$A\approx B$ denotes that $A\lesssim B$ and $B\lesssim A$.
Throughout this paper, the letter $C$ will be used for positive
constants independent of relevant variables that may change from
one occurrence to another.
\section{Some Lemmas}\label{s2}
We begin with some properties of weighted amalgam spaces in this section.
\begin{lemma}\label{DuiO}\cite{Duiou}
Give $1\leq p,q<\infty$,
$$
[(L^p_w,L^q_v)_t(\mathbb R^n)]'=(L^{p'}_{w'},L^{q'}_{v'})_t(\mathbb R^n),
$$
where $\frac{1}{p}+\frac{1}{p'}=\frac{1}{q}+\frac{1}{q'}=1$, $v'=v^{1-p'}$, $w'=w^{1-q'}$,
and as for dual space of the weighted amalgam space, then we know
$$
[(L^p_w,L^q_v)_t(\mathbb R^n)]':=\left\{f:\,\|f\|_{[(L^p_w,L^q_v)_t(\mathbb R^n)]'}
:=\sup_{\|g\|_{(L^p_w,L^q_v)_t(\mathbb R^n)}\leq1}\int_{\mathbb R^n} f(t)g(t)dt\right\}.
$$
\end{lemma}
\begin{lemma}\label{Holder}\cite{Duiou}
Given $1\leq p,q\leq\infty$,
$$
\|fg\|_{L^1(\mathbb R^n)}\leq\|f\|_{(L^p_w,L^q_v)_t(\mathbb R^n)}\|g\|_{(L^{p'}_{w'},L^{q'}_{v'})_t(\mathbb R^n)},
$$
where $\frac{1}{p}+\frac{1}{p'}=\frac{1}{q}+\frac{1}{q'}=1$, $v'=v^{1-p'}$, $w'=w^{1-q'}$.
\end{lemma}

\begin{lemma}\label{waicha}\cite{refwcha}
Let $\mathcal F$ be a given family of pairs $(f,g)$ of non-negative and not identically zero
measurable functions on $\mathbb R^n$.
Suppose that for some fixed exponent $p_0\in[1, \infty)$,
and every weight $\omega\in A_{p_0}$,
$$
\int_{\mathbb R^n}f(x)^{p_0}\omega(x)dx\leq C_{\omega,p_0}\int_{\mathbb R^n}g(x)^{p_0}\omega(x)dx,
\quad\forall (f,g)\in \mathcal F.
$$
Then, for all $p\in (1,\infty)$ and for all $\omega \in A_p$,
$$
\int_{\mathbb R^n}f(x)^p\omega(x)dx\leq C_{\omega,p}\int_{\mathbb R^n}g(x)^p\omega(x)dx,
\quad\forall (f,g)\in \mathcal F.
$$
\end{lemma}

\begin{lemma}\label{Pro}
Let $t\in(0,\infty)$, $r\in[1,\infty)$ and $1<\alpha,p<\infty$. Given $w\in A_p$,
$v\in A_q$, and $q\in [1,rp]$, there exist a positive constant $C$
such that
\begin{align*}
C^{-1}||f||_{(L_{w}^{p},L_{v}^{q})_t(\mathbb R^n)}\leq ||f||_{(L_{w}^{p},L_{v}^{q})_{\alpha t}(\mathbb R^n)}
\leq C\alpha ^{np}||f||_{(L_{w}^{p},L_{v}^{q})_t(\mathbb R^n)}.
\end{align*}
Where $C$ is independent of $f,t,\alpha$.
\end{lemma}
\begin{proof}[\bf{Proof of Lemma \ref{Pro}}]
Similar to the proof of \cite[Theorem 3.6]{W}, we only prove the right-hand of the above inequality.
We split it to three steps. We first obtain the case $q = p$ and $1\leq r< \infty$.
From this, we extrapolate concluding the desired estimate in the ranges $1\leq q \le rp$ and
$1< r< \infty$. Finally, we shall consider the case $r=1$ and $1\leq q < p$.
From now on, we fix $\alpha>1$. For the first step, let $q=p$ and $1\leq r<\infty$.
By $\omega\in A_p$ and (\ref{Eqw}), we conclude that
\begin{equation}\label{Eqa}
\begin{split}
\int_{\mathbb R^n}\frac{1}{w(B(x,\alpha t))}\int_{B(x,\alpha t)}|f(y)|^pw(y)dyv(x)dx
    &=\frac{1}{w(B(x,\alpha t))}\int_{\mathbb R^n}|f(y)|^pw(y)v(B(y,\alpha t)dy\\
    &\leq \frac{\alpha ^{np}}{w(B(x,\alpha t))}\int_{\mathbb R^n}|f(y)|^pw(y)v(B(y, t))dy\\
    &\leq \alpha ^{np} \int_{\mathbb R^n}\frac{1}{w(B(x,t))}\int_{B(x,t)}|f(y)|^pw(y)dyv(x)dx.
\end{split}
\end{equation}
To prove the second step, take an arbitrary $p_0\in[1,\infty)$ and consider $\mathcal F$ the family of pairs
$$
(f,g):=\lt(\lt[\frac{1}{w(B(x,\alpha t))}\int_{B(x,\alpha t)}|f(y)|^pw(y)dy\rt]^\frac{1}{p_0},
\lt[\frac{1}{w(B(x,t))}\int_{B(x,t)}|f(y)|^pw(y)dy\rt]^\frac{1}{p_0}\rt).
$$
Then, for any $v\in A_{p_0}$, (\ref{Eqa}) gives
\begin{align*}
\int_{\mathbb R^n}f(x)^{p_0}v(x)dx&=\int_{\mathbb R^n}\frac{1}{w(B(x,\alpha t))}
    \int_{B(x,\alpha t)}|f(y)|^pw(y)dyv(x)dx\\
    &\le \alpha ^{np_0} \int_{\mathbb R^n}\frac{1}{w(B(x,t))}\int_{B(x,t)}|f(y)|^pw(y)dyv(x)dx\\
    &\le \alpha ^{np_0}\int_{\mathbb R^n}g(x)^{p_0}v(x)dx.
\end{align*}
Applying Lemma \ref{waicha}, then, for any given $1 < r < \infty$ and $v\in A_r$,
$$
\int_{\mathbb R^n}\lt[\frac{1}{w(B(x,\alpha t))}\int_{B(x,\alpha t)}|f(y)|^pw(y)dy
    \rt]^{\frac{r}{p_0}}v(x)dx
    \leq C\alpha ^{nr}\int_{\mathbb R^n}\lt[\frac{1}{w(B(x,t))}
    \int_{B(x,t)}|f(y)|^pw(y)dy\rt]^{\frac{r}{p_0}}v(x)dx.
$$
From this and letting $p_0 =\frac{pr}{q}$, we prove Lemma \ref{Pro} under the restriction $r\in (1,\infty)$,
that is, whenever $r\in (1,\infty)$, $v\in A_r$ and $q\in [1, rp]$,
there exists a positive constant $C$ such that the inequality
$$
\int_{\mathbb R^n}\lt[\frac{1}{w(B(x,\alpha t))}\int_{B(x,\alpha t)}|f(y)|^pw(y)dy
    \rt]^{\frac{q}{p}}v(x)dx
    \leq C\alpha ^{nr}\int_{\mathbb R^n}\lt[\frac{1}{w(B(x,t))}
    \int_{B(x,t)}|f(y)|^pw(y)dy\rt]^{\frac{q}{p}}v(x)dx.
$$
holds true.
\par Then, we consider the case $r = 1$ and $1\leq q < p$. Without loss
of generality, we assume that
$$
\int_{\mathbb R^n}\lt[\frac{1}{w(B(x,t))}\int_{B(x,t)}|f(y)|^pw(y)dy\rt]^{\frac{q}{p}}v(x)dx<\infty.
$$
For a fixed $\lambda> 0$, set
$$
E_{\lambda}:=\lt\{ x\in\mathbb R^n:\lt[\frac{1}{w(B(x,t))}\int_{B(x,t)}|f(y)|^pw(y)dy\rt]^\frac1p\le\lambda\rt\}
$$
and
$$
O_{\lambda}:=\mathbb R^n\backslash E_{\lambda}=\lt\{ x\in\mathbb R^n:\lt[\frac{1}{w(B(x,t))}
\int_{B(x,t)}|f(y)|^pw(y)dy\rt]^\frac1p> \lambda\rt\}.
$$
Then, for each $0 <\gamma< 1$, we also consider the set of global $\gamma$-density with respect to $E_\gamma$ defined
by
$$
E^*_\lambda:=\left\{ x\in \mathbb R^n:\frac{|E_\lambda\cap B|}{|B|}\ge\lambda,
\forall B~ centered~ at~ x\right\}
$$
and denote its complement by
$$
O^*_\lambda:=\left\{ x\in \mathbb R^n: \exists r>0~s.~t.~\frac{|O_\lambda\cap B(x,r)|}{|B(x,r)|}> 1- \gamma\right\}
=\left\{x\in\mathbb R^n: M(\chi_{O_{\lambda}})(x)> 1-\gamma \right\}.
$$
By \cite{MP}, it can prove that $E^*_\lambda$ is closed and $\varnothing \not\subseteq E^*_\lambda\subset E_\lambda$.
\par Denote by $\Re(E^*_\lambda):=\underset{x\in E^*_\lambda}{\mathop{\cup }}\,\{{y\in\mathbb R^n:|y-x|<\alpha t}\}$.
For any $y\in\Re(E^*_\lambda)$, there exists a $\bar{x}\in E^*_\lambda$ such that $|\bar{x}-y|<\alpha t$.
Let $z=y-\frac{t}{2}\frac{y-\bar{x}}{|\bar{x}-y|}$. Then $B(z,\frac{t}{2})\subset B(\bar{x},\alpha t)\bigcap B(y,t)$
and
$$
|B(\bar x,\alpha t)\backslash B(y,t)|\leq \lt|B(\bar{x},\alpha t)\backslash B\lt(z,\frac{t}{2}\rt)\rt|
=|B(\bar{x},\alpha t)|-\lt|B\lt(z,\frac{t}{2}\rt)\rt|=\left(1-\frac{1}{2^n\alpha^n}\right)|B(\bar{x},\alpha t)|.
$$
As for $\bar{x}\in E^*_\lambda$. Then we obtain that
\begin{align*}
\gamma|B(\bar{x},\alpha t)|\leq |E_\lambda \cap B(\bar{x},\alpha t)|
    &=|E_\lambda \cap B(\bar{x},\alpha t)\backslash B(y,t)|+|E_\lambda \cap B(\bar{x},\alpha t)\cap B(y,t)|\\
    &\leq \left(1-\frac{1}{2^n\alpha^n}\right)|B(\bar{x},\alpha t)|+|E_\lambda \cap B(y,t)|.
\end{align*}
Choosing $\gamma =1-\frac{1}{2^{n+1}\cdot\alpha^n}$ yields
\begin{equation}\label{E}
|E_\lambda \cap B(y,t)|\ge \frac{1}{2^{n+1}\cdot\alpha^n}|B(\bar{x},\alpha t)|,
\end{equation}
which, together with (\ref{Eqw}) by Proposition \ref{Eqweight}, further implies that,
for any $y\in\Re(E^*_\lambda)$,
\begin{equation}\label{weightE}
\frac{v(E_\lambda\cap B(y,t))}{v(B(y,\alpha t))}\ge C\frac{|E_\lambda\cap B(y,t)|}{|B(\bar{x},\alpha t)|}
\ge C\frac{1}{2^{n+1}\alpha^n}.
\end{equation}
It follows from (\ref{weightE}) that
\begin{equation}\label{F}
\begin{split}
\int_{E^*_\lambda}\frac1{w(\alpha B)}\int_{B(x,t)}|f(y)|^pw(y)dyv(x)dx
    &=\frac{1}{w(B(x,\alpha t))}\int_{E^*_\lambda}\int_{\mathbb R^n}|f(y)|^p\chi_{B(0,1)}
    (\frac{y-x}{\alpha t})v(x)w(y)dydx\\
    &=\frac{1}{w(B(x,\alpha t))}\int_{\Re(E^*_\lambda)}|f(y)|^pw(y)\int_{B(y,\alpha t)}v(x)dxdy\\
    &\leq C_v\frac{2^{n+1}\alpha^n}{w(B(x,\alpha t))}\int_{\Re(E^*_\lambda)}
|f(y)|^pw(y)\int_{B(y,t)\cap E_\lambda} v(x)dxdy\\
    &\leq C_v2^{n+1}\alpha^n\int_{E_\lambda}\frac{1}{w(B(x,t))}\int_{B(x,t)}|f(y)|^pw(y)dyv(x)dx.
\end{split}
\end{equation}
Since $v\in A_1$,
by Lemma \ref{ProMf}, \eqref{E}, \eqref{F} and the definition of $O^*_\lambda$, then
\begin{align*}
    v&\lt(\lt\{x\in\mathbb R^n:\lt[\frac{1}{w(B(x,\alpha t))}\int_{B(x,t)}|f(y)|^pw(y)dy
\rt]^{\frac{1}{p}}>\lambda \rt\}\rt)\\
    &\leq v\left(\left\{x\in O^*_\lambda:\left[\frac{1}{w(B(x,\alpha t))}\int_{B(x,t)}|f(y)|^pw(y)dy
\right]^{\frac{1}{p}}>\lambda \right\}\right)\\
    &\quad+\left(\left\{x\in E^*_\lambda:\left[\frac{1}{w(B(x,\alpha t))}\int_{B(x,t)}|f(y)|^pw(y)dy
\right]^{\frac{1}{p}}>\lambda \right\}\right)\\
    &\leq v (\left\{x\in\mathbb R^n: M(\chi_{O_{\lambda}})(x)> 1-\gamma\right\})
+\frac{1}{\lambda^p}\int_{E^*_\lambda} \frac{1}{w(B(x,\alpha t))}\int_{B(x,t)}|f(y)|^pw(y)dyv(x)dx\\
    &\leq C_{n,v}\left[v(O_\lambda)+\frac{1}{\lambda^p}\int_{E_\lambda}
\frac{1}{w(B(x,\alpha t))}\int_{B(x,t)}|f(y)|^pw(y)dyv(x)dx\right]\\
    &\leq C_{n,v}v\left(\left\{x\in\mathbb R^n:\left[
\frac{1}{w(B(x,t))}\int_{B(x,t)}|f(y)|^pw(y)dy\right]^{\frac{1}{p}}>\lambda\right\}\right)\\
    &\quad+C_{n,v}\left(\frac{1}{\lambda^p}\int_{E_\lambda}
\frac{1}{w(B(x,t))}\int_{B(x,t)}|f(y)|^pw(y)dyv(x)dx\right).
\end{align*}
Using this and the assumption that $1\leq q < p$, then
\begin{align*}
&\int_{\mathbb R^n}\left[\frac{1}{w(B(x,\alpha t))}\int_{B(x,t)}|f(y)|^pw(y)dy\right]^{\frac{q}{p}}v(x)dx\\
    &=\int_{0}^{\infty}q\lambda^qv\left(\left\{x\in\mathbb R^n:\left[\frac{1}{w(B(x,\alpha t))}
\int_{B(x,\alpha t)}|f(y)|^pw(y)dy\right]^{\frac{1}{p}}>\lambda\right\}\right)\frac{d\lambda}{\lambda}\\
    &\leq C_{n,v}\int_{0}^{\infty}q\lambda^qv\left(\left\{x\in\mathbb R^n:\left[\frac{1}{w(B(x,t))}
\int_{B(x,t)}|f(y)|^pw(y)dy\right]^{\frac{1}{p}}>\lambda\right\}\right)\frac{d\lambda}{\lambda}\\
    &\quad+C_{n,v}\int_{0}^{\infty}q\lambda^{q-\lambda}\int_{E_\lambda}\frac{1}{w(B(x,t))}
\int_{B(x,t)}|f(y)|^pw(y)dyv(x)dx\frac{d\lambda}{\lambda}\\
    &\leq C_{n,v}\int_{\mathbb R^n}\left[\frac{1}{w(B(x,t))}
\int_{B(x,t)}|f(y)|^pw(y)dy\right]^{\frac{q}{p}}v(x)dx\\
    &\quad+C_{n,v}\int_{\mathbb R^n}\frac{1}{w(B(x,t))}\int_{B(x,t)}|f(y)|^pw(y)dy\int_{\left[\frac{1}{w(B(x,\alpha t))}
\int_{B(x,\alpha t)}|f(y)|^pw(y)dy\right]^\frac{1}{p}}^{\infty}
q\lambda^{q-\lambda}\frac{d\lambda}{\lambda}v(x)dx\\
    &\leq C_{n,v}\int_{\mathbb R^n}\left[\frac{1}{w(B(x,t))}
\int_{B(x,t)}|f(y)|^pw(y)dy\right]^{\frac{q}{p}}v(x)dx.
\end{align*}
Thus, we complete the proof of Lemma $\ref{Pro}$.
\end{proof}

\begin{lemma}\cite{15}\label{ProMf}
Let $1<p\leq\infty$ and $w\in A_{p}$, then there exist a positive constant $C$ such that
$$
\lt(\int_{\mathbb R^n}|Mf(x)|^pw(x)dx\rt)^\frac{1}{p}\le C\lt(\int_{\mathbb R^n}|f(x)|^pw(x)dx\rt)^\frac{1}{p}.
$$
For all $\lambda>0$, if $p=1$, and $w\in A_{1}$, then there exist a positive constant $C$ such that
$$
w(\{x\in \mathbb R^n:|M(f)(x)|>\lambda\})\le \frac{C}{\lambda}\int_{\mathbb R^n}|f(x)|w(x)dx.
$$
The universal positive constant $C$ is independent of $f$ and $\lambda$.
\end{lemma}
\begin{lemma}\cite{8}\label{ProTf}
Let $1<p<\infty$ and $w\in A_{p}$, then there exist a positive constant $C$ such that
$$
\lt(\int_{\mathbb R^n}|Tf(x)|^pw(x)dx\rt)^\frac{1}{p}\le C\lt(\int_{\mathbb R^n}|f(x)|^pw(x)dx\rt)^\frac{1}{p}.
$$
For all $\lambda>0$, if $p=1$, and $w\in A_{1}$, then there exist a positive constant $C$ such that
$$
w\left(\{x\in \mathbb R^n:|T(f)(x)|>\lambda\}\right)\le
\frac{C}{\lambda}\int_{\mathbb R^n}|f(x)|w(x)dx.
$$
The universal positive constant $C$ is independent of $f$ and $\lambda$.
\end{lemma}

\begin{lemma}\label{ProIf}\cite{16}
Let $0<\alpha<n$, $1<p<\frac n\alpha$, $\frac{1}{p}-\frac{1}{q}=\frac{\alpha}{n}$, and $w\in A_{p,q}$,
then there exist a positive constant $C$ such that
$$
\left(\int_{\mathbb R^n}|I_{\alpha}f(x)w(x)|^qdx\right)^\frac{1}{q}
    \le C\left(\int_{\mathbb R^n}|f(x)w(x)|^pdx\right)^\frac{1}{p}.
$$
If $p=1$, and $w\in A_{1,q}$ with $q=\frac n{n-\alpha}$, then for all $\lambda>0$,
then there exist a positive constant $C$ such that
$$
w\left(\{x\in \mathbb R^n:|I_{\alpha}(f)(x)|>\lambda \}\right)
    \leq C \lt(\frac1{\lambda}\int_{\mathbb R^n}|f(x)|w(x)^\frac 1q dx\rt)^q.
$$
The universal positive constant $C$ is independent of $f$ and $\lambda$.
\end{lemma}

\section{The Proofs of Main Theorems}\label{s3}
In this section, the proofs of these theorems are given.
We first recall the definition of the uncentered Hardy--Littlewood maximal operator,
for almost every $x\in\mathbb R^n$,
$$
\overline{M}f(x):=\sup_{B\ni x}\frac {1}{|B|}\int_B|f(y)|dy.
$$
Then, we have the following pointwise inequality.
\begin{lemma}\label{LeMf}
For all $x\in$ $\mathbb R^n$, $t>0$, and $1<r< \infty$, and all $f$ locally $r$ integrable, then
\begin{equation}\label{EqMf}
\begin{split}
\left(\frac{1}{w(B(x,t))}\int_{B(x,t)}{{{| M(f)(y)|}^{r}}}w(y)dy\right)^{\tfrac{1}{r}}\lesssim \left(\frac{1}{w(B(x,2t))}{\int_{B(x,2t)}|f(y)|^r}w(y)dy\right)^{\tfrac{1}{r}}\\
+\overline{M}\left(\frac{1}{w(B(\cdot ,t))}\int_{B(\cdot ,t)}| f(z)|^rw(z)dz\right)^\frac{1}{r}\lt(x\rt).
\end{split}
\end{equation}
\end{lemma}
\begin{proof}[\bf{Proof of Lemma \ref{LeMf}}]
Fix $x \in \mathbb R^n$ and $t>0$, and split the sumprem into $0<\tau\le t$ and $t<\tau$, and then
\begin{align*}
&\left(\frac{1}{w(B(x,t))}\int_{B(x,t)}{{{| M(f)(y)|}^{r}}}w(y)dy\right)^{\tfrac{1}{r}}\\
   &\quad \le \left(\frac{1}{w(B(x,t))}\int_{B(x,t)}\left(\sup_{\substack{0<\tau\le t}}\frac{1}{|B(y,\tau)}\int_{B(y,\tau)}
|f(z)|dz\right)^rw(y)dy\right)^\frac{1}{r}\\
    &\qquad+\left(\frac{1}{w(B(x,t))}\int_{B(x,t)}\left(\sup_{\substack{\tau>t}} \frac{1}{|B(y,\tau)}\int_{B(y,\tau)}
|f(z)|dz\right)^rw(y)dy\right)^\frac{1}{r}
    :=\uppercase\expandafter{\romannumeral1}+\uppercase\expandafter{\romannumeral2}.	
\end{align*}
For $I$, since $y \in B(x,t)$,  $B(y,\tau)\subset B(x,2t)$. Then it follows from Lemma \ref{ProMf} that
\begin{align*}
I&= \left(\frac{1}{w(B(x,t))}\int_{B(x,t)}\left(\sup_{\substack{0<\tau\le t}}\frac{1}{|B(y,\tau)|}\int_{B(y,\tau)} |f(z)|dz\right)^rw(y)dy\right)^\frac{1}{r}\\
    &\le \left(\frac{1}{w(B(x,t))}\int_{B(x,t)}|M(f\cdot \chi_{B(x,2t)} )(y)|^rw(y)dy)\right)^\frac{1}{r}\\
    &\lesssim \left(\frac{1}{w(B(x,2t))}\int_{B(x,2t)}|f(y)|^rw(y)dy\right)^\frac{1}{r}.
\end{align*}
For $II$, for any $z,\xi \in \mathbb R^n$, $\xi \in B(z,t)$ is equivalent to $z\in B(\xi,t)$. If $z\in B(y,\tau)$, $\xi \in B(z,t)$, then $\xi \in B(y,2\tau)$. Besides, owing to $x\in B(y,t)$, then $x\in B(y,2\tau)$. Applying the Fubini's theorem and H\"older's inequality, then we get
\begin{align*}
\uppercase\expandafter{\romannumeral2}
    &=\lt(\frac{1}{w(B(x,t))}\int_{B(x,t)}\lt(\sup_{\substack{\tau>t}} \frac{1}{|B(y,\tau)|}\int_{B(y,\tau)}
|f(z)|dz\rt)^rw(y)dy\rt)^\frac{1}{r}\\
    &\le \lt(\frac{1}{w(B(x,t))}\int_{B(x,t)}\lt(\sup_{\substack{\tau>t}} \frac{1}{|B(y,\tau)|}\int_{B(y,\tau)}
|f(z)|{\frac{1}{|B(z,t)|}}\int_{B(z,t)}d\xi dz\rt)^rw(y)dy\rt)^\frac{1}{r}\\
    &\le \lt(\frac{1}{w(B(x,t))}\int_{B(x,t)}\lt(\sup_{\substack{\tau>t}} \frac{1}{|B(y,2\tau)|}\int_{B(y,\tau)}
{\frac{1}{|B(\xi,t)|}}\int_{B(\xi,t)}|f(z)|dz d\xi\rt)^rw(y)dy\rt)^\frac{1}{r}\\
    &\le\overline{M}\lt(\frac{1}{|B(\cdot,t)|}\int_{B(\cdot,t)}|f(z)|dz\rt)(x)\\
    &\le\overline{M}\lt(\lt(\frac{1}{w(B(\cdot,t))}\int_{B(\cdot,t)}|f(z)|^rw(z)dz\rt)^\frac{1}{r}\rt)(x).
\end{align*}
Thus, we complete the proof of Lemma \ref{LeMf}.
\end{proof}

\begin{proof}[\bf{Proof of Theorem \ref{ThMf}}]
We first prove the case $p\in(1,\infty)$ and $q\in(1,\infty]$. By Lemma \ref{Pro}, Lemma\ref{ProMf},
and  Lemma \ref{LeMf}, then
\begin{align*}
  ||M(f)||_{(L_w^p,L_v^q)_t(\mathbb R^n)}&={\left\|{\left(\frac{1}{w(B(\cdot,t))}\int_{B(\cdot,t)}| M(f)(y)|^p
  w(y)dy\right)^{\frac{1}{p}}}\right\|}_{L^q_v(\mathbb R^n)}\\
&\lesssim \left\|\left(\frac{1}{w(B(\cdot,2t))}\int_{B(\cdot,2t)}{|f(y)|}^{p}w(y)dy\right)^\frac{1}{p}
\right\|_{L_{v}^q(\mathbb R^n)}+\left\|\overline{M}\left(\frac{1}{w(B(\cdot ,t))}\int_{B(\cdot ,t)}
   |f(z)|^pw(z)dz\right)^\frac{1}{p}\right\|_{L^q_v(\mathbb R^n)} \\
&\lesssim ||f||_{(L_w^p,L_v^q)_t(\mathbb R^n)} +\left\|\left(\frac{1}{w(B(\cdot ,t))}\int_{B(\cdot ,t)}
|f(z)|^pw(z)dz\right)^\frac{1}{p}\right\|_{L^q_v(\mathbb R^n)}\lesssim {||f||}_{(L_w^p,L_v^q)_t(\mathbb R^n)}.
\end{align*}	
\par And then for the case of $p>1$ and $q=1$.\\
By Lemma \ref{DuiO}, there exists $g\in (L^{p'}_{w'},L^{\infty}_{v'})_t(\mathbb R^n)$ such that
$$
\left\|\chi_{\{x\in\mathbb R^n:|Mf(x)|>\lambda\}}\right\|_{(L^p_w, L^1_v)_t(\mathbb R^n)}
    =\int_{\mathbb R^n}\chi_{\{x\in\mathbb R^n:|Mf(x)|>\lambda\}}(x)g(x)dx.
$$
Since $g(x)\leq [M(|g|^\frac{1}{\gamma})]^\gamma(x)$ and $[M(|g|^\frac{1}{\gamma})]^\gamma(x)\in A_1$ for $\gamma>1$,
by Lemma \ref{ProMf}, then it can obtain that
\begin{align*}
\lambda\left\|\chi_{\{x\in\mathbb R^n:|Mf(x)|>\lambda\}}\right\|_{(L^p_w, L^1_v)_t(\mathbb R^n)}
    &\leq \lambda\int_{\mathbb R^n}\chi_{\{x\in\mathbb R^n:|Mf(x)|>\lambda|\}}(x)
[M(|g|^\frac{1}{\gamma})]^\gamma(x)dx\\
    &\leq C\int_{\mathbb R^n}|f(x)| [M(|g|^\frac{1}{\gamma})]^\gamma(x)dx.
\end{align*}
By taking the supremum over all $\lambda>0$, then we get
$$
\|Mf\|_{W(L^p_w,L^1_v)(\mathbb R^n)}\leq C\int_{\mathbb R^n}|f(x)| [M(|g|^\frac{1}{\gamma})]^\gamma(x)dx.
$$
By Lemma \ref{Holder} and the fact $M$ is bounded on $(L^{p'}_{w'}, L^{\infty}_{v'})_t(\mathbb R^n)$,
$$
\|Mf\|_{W(L^p_w, L^1_v)_t(\mathbb R^n)}\leq C\|f\|_{(L^p_w, L^1_v)_t(\mathbb R^n)}
    \|[M(|g|^\frac{1}{\gamma})]^\gamma\|_{(L^{p'}_{w'},L^{\infty}_{v'})_t(\mathbb R^n)}
\leq \|f\|_{(L^p_w, L^1_v)_t(\mathbb R^n)}
    \|g\|_{(L^{p'}_{w'},L^{\infty}_{v'})_t(\mathbb R^n)}.
$$
Hence,
$$
\|Mf\|_{W(L^p_w, L^1_v)_t(\mathbb R^n)}\leq C\|f\|_{(L^p_w, L^1_v)_t(\mathbb R^n)}.
$$
This completes the proof of the Theorem 1.1.
\end{proof}

\begin{remark}
For any $x\in\mathbb R^n\backslash\{0\}$,
$$
Hf(x):=\frac1{|x|^n}\int_{|y|\leq |x|}|f(y)|dy=\frac C{|B(0,|x|)|}\int_{B(0,|x|)}|f(y)|dy\leq CMf(x).
$$
By Theorem \ref{ThMf}, we conclude that the Hardy operator $H$ is bounded on $(L^p_w,L^q_v)_t(\mathbb R^n)$
for $p,q>1$.
\end{remark}

\begin{proof}[\bf{Proof of theorem \ref{ThTf}}]
Let $p=r\theta$, $q=s\theta$, where $\theta>1$, $r>1$ and $s>1$. Then, by Lemma \ref{DuiO},
for $g\in (L^{p'}_{w'},L^{q'}_{v'})_t(\mathbb R^n)$, then
$$
\|Tf\|_{(L^p_w, L^q_v)_t(\mathbb R^n)}=\left\||Tf|^\theta\right\|^\frac{1}{\theta}_{(L^r_w, L^s_v)_t(\mathbb R^n)}
=\left(\int_{\mathbb R^n}|Tf(x)|^\theta g(x)dx\right)^\frac{1}{\theta}.
$$
Let $w:=[M(|g|^\frac{1}{\gamma})]^\gamma(x)$. The fact $[M(|g|^\frac{1}{\gamma})]^\gamma(x)\in A_1$,
$g(x)\leq [M(|g|^\frac{1}{\gamma})]^\gamma(x)$,
and Lemma \ref{ProTf} yield
$$
\|Tf\|_{(L^p_w, L^q_v)_t(\mathbb R^n)}\leq \left(\int_{\mathbb R^n}|Tf(x)|^\theta
    [M(|g|^\frac{1}{\gamma})]^\gamma (x)dx\right)^\frac{1}{\theta}
    \leq C\left(\int_{\mathbb R^n}|f(x)|^\theta
    [M(|g|^\frac{1}{\gamma})]^\gamma(x) dx\right)^\frac{1}{\theta}.
$$
By Lemma \ref{Holder} and Theorem \ref{ThMf}, thus,
$$
\|Tf\|_{(L^p_w, L^q_v)_t(\mathbb R^n)}\lesssim\||f|^\theta\|_{(L^r_w, L^s_v)_t(\mathbb R^n)}^\frac{1}{\theta}
    \|[M(|g|^\frac{1}{\gamma})]^\gamma\|_{(L^{r'}_{w'},L^{s'}_{v'})_t(\mathbb R^n)}
\lesssim\|f\|_{(L^p_w, L^q_v)_t(\mathbb R^n)}\|g\|_{(L^{p'}_{w'},L^{q'}_{v'})_t(\mathbb R^n)}.
$$
And hence we complete the proof for the case of $p,q>1$.

Now, we consider $p>1$ and $q=1$.
\par Taking $g\in (L^{p'}_{w'},L^{\infty}_{v'})_t(\mathbb R^n)$,
by Lemma \ref{DuiO}, we can write
$$
\left\|\chi_{\{x\in\mathbb R^n:|Tf(x)|>\lambda\}}\right\|_{(L^p_w, L^1_v)_t(\mathbb R^n)}
    =\int_{\mathbb R^n}\chi_{\{x\in\mathbb R^n:|Tf(x)|>\lambda\}}(x)g(x)dx.
$$
As for $g(x)\leq [M(|g|^\frac{1}{\gamma})]^\gamma(x)$ with $\gamma>1$, and by Lemma \ref{ProTf}, then we know
\begin{align*}
\lambda\left\|\chi_{\{x\in\mathbb R^n:|Tf(x)|>\lambda\}}\right\|_{(L^p_w, L^1_v)_t(\mathbb R^n)}
    &\leq \lambda\int_{\mathbb R^n}\chi_{\{x\in\mathbb R^n:|Tf(x)|>\lambda|\}}(x)
[M(|g|^\frac{1}{\gamma})]^\gamma(x)dx\\
    &\leq C\int_{\mathbb R^n}|f(x)| [M(|g|^\frac{1}{\gamma})]^\gamma(x)dx.
\end{align*}
Finally, take the supremum over $\lambda>0$, hence, it can show that
$$
\|Tf\|_{W(L^p_w,L^1_v)(\mathbb R^n)}\leq C\int_{\mathbb R^n}|f(x)| [M(|g|^\frac{1}{\gamma})]^\gamma(x)dx.
$$
From Lemma \ref{Holder} and Theorem \ref{ThMf}, hence,
\begin{align*}
\|Tf\|_{W(L^p_w, L^1_v)_t(\mathbb R^n)}
    &\leq C\|f\|_{(L^p_w, L^1_v)_t(\mathbb R^n)}
\|[M(|g|^\frac{1}{\gamma})]^\gamma\|_{(L^{p'}_{w'},L^{\infty}_{v'})_t(\mathbb R^n)}\\
    &\leq \|f\|_{(L^p_w, L^1_v)_t(\mathbb R^n)}
\|g\|_{(L^{p'}_{w'},L^{\infty}_{v'})_t(\mathbb R^n)}.
\end{align*}
Hence, $\|Tf\|_{W(L^p_w, L^1_v)_t(\mathbb R^n)}\leq C\|f\|_{(L^p_w, L^1_v)_t(\mathbb R^n)}$,
the result holds.
\end{proof}

\begin{proof}[\bf{Proof of Theorem \ref{ThIf}}]
For $p_0,\,q_0>1$, let
$$
\frac{1}{p}=\frac{1}{p_0}-\frac{\alpha}{n},\quad \frac{1}{q}=\frac{1}{q_0}-\frac{\alpha}{n}.
$$
For $\theta>1$ such that $\frac{p}{\theta}>1$ and $\frac{q}{\theta}>1$. Then,
for $g\in (L^{(p/\theta)'}_{w'},L^{(q/\theta)'}_{v'})_t(\mathbb R^n)$, by Lemma \ref{DuiO},
then we can obtain that
$$
    \|I_\alpha f\|_{(L^p_w, L^q_v)_t(\mathbb R^n)}=\left\||I_\alpha f|^\theta
\right\|^\frac{1}{\theta}_{(L^{p/\theta}_w, L^{q/\theta}_v)_t(\mathbb R^n)}
    =\left(\int_{\mathbb R^n}|I_\alpha f(x)|^\theta g(x)dx\right)^\frac{1}{\theta}.
$$
Noticing that for $0<\eta<1$, $M_\eta(|g|)(x):=[M(|g|^\eta)]^\frac{1}{\eta}(x)\in A_1$,
and letting $w:=[M_\eta(|g|)(x)]^\frac{1}{\theta}$, we have $w^\theta\in A_1$
and hence $w^\theta\in A_{\theta\frac{n-\alpha}{n}}$. Denote $\frac{1}{\gamma}=\frac{1}{\theta}+\frac{\alpha}{n}$.
Then $w\in A_{\gamma,\theta}$. Moreover, by Lemmas \ref{Holder} and \ref{ProIf}, then
\begin{align*}
\|I_\alpha f\|_{(L^p_w, L^q_v)_t(\mathbb R^n)}
    &\leq \lt[\int_{\mathbb R^n}|I_\alpha f(x)|^\theta M_\eta(|g|)(x)dx\rt]^\frac1\theta
    \leq C\lt[\int_{\mathbb R^n}|f(x)|^\gamma[w(x)]^\gamma dx\rt]^\frac1\gamma\\
    &\leq C\lt[\||f|^\gamma\|_{L^{p/\theta}_w, L^{q/\theta}_v)_t(\mathbb R^n)}
\|[M_\eta(|g|)]^\frac\gamma\theta\|_{(L^{(p/\theta)'}_{w'},L^{(q/\theta)'}_{v'})_t(\mathbb R^n)}\rt]^\frac 1\gamma\\
    &=C\|f\|_{(L^{p_0}_w, L^{q_0}_v)_t(\mathbb R^n)}
\|[M_\eta(|g|)]^\frac\gamma\theta\|_{(L^{(p_0/\gamma)'}_{w'},L^{(q_0/\gamma)'}_{v'})_t(\mathbb R^n)}^\frac 1\gamma.
\end{align*}
Applying Theorem \ref{ThMf}, then
\begin{align*}
\|[M_\eta(|g|)]^\frac\gamma\theta\|_{(L^{(p_0/\gamma)'}_{w'},L^{(q_0/\gamma)'}_{v'})_t(\mathbb R^n)}^\frac 1\gamma
    &=\|[M(|g|^\eta)]^\frac\gamma{\eta\theta}\|
_{(L^{(p_0/\gamma)'}_{w'},L^{(q_0/\gamma)'}_{v'})_t(\mathbb R^n)}^\frac 1\gamma\\
    &=\|M(|g|^\eta)\|_{(L^{(p_0/\gamma)'\frac\gamma{\eta\theta}}_{w'},
L^{(q_0/\gamma)'\frac\gamma{\eta\theta}}_{v'})_t(\mathbb R^n)}^\frac 1{\eta\theta}\\
    &\leq C\||g|^\eta\|_{(L^{(p_0/\gamma)'\frac\gamma{\eta\theta}}_{w'},
L^{(q_0/\gamma)'\frac\gamma{\eta\theta}}_{v'})_t(\mathbb R^n)}^\frac 1{\eta\theta}\\
    &=C\|g\|_{(L^{(p_0/\gamma)'\frac\gamma{\theta}}_{w'},
L^{(q_0/\gamma)'\frac\gamma{\theta}}_{v'})_t(\mathbb R^n)}^\frac 1\theta.
\end{align*}
It is easy to check that
$$
\frac 1{(p_0/\gamma)'\frac \gamma\theta}=(1-\frac\gamma{p_0})\frac\theta\gamma
=(1-\frac\gamma p-\frac{\gamma\alpha}n)\frac\theta\gamma=\theta(\frac 1\gamma+\frac\alpha n)
=1-\frac1 {p/\theta}=\frac1{(p/\theta)'}
$$
and
$$
\frac 1{(q_0/\gamma)'\frac \gamma\theta}=(1-\frac\gamma{q_0})\frac\theta\gamma
=(1-\frac\gamma q-\frac{\gamma\alpha}n)\frac\theta\gamma=\theta(\frac 1\gamma+\frac\alpha n)
=1-\frac1 {q/\theta}=\frac1{(q/\theta)'}.
$$
So
$$
\|I_\alpha f\|_{(L^p_w, L^q_v)_t(\mathbb R^n)}\leq C\|f\|_{(L^{p_0}_w, L^{q_0}_v)_t(\mathbb R^n)}.
$$
Then, we consider the case of theorem \ref{ThIf} (b).
\par For $p_0>1$ and $q_0=1$, let
$$
\frac{1}{p}=\frac{1}{p_0}-\frac{\alpha}{n},\quad \frac{1}{q}=1-\frac{\alpha}{n}.
$$
Take $\theta=q=\frac n{n-\alpha}$. Then,
for $g\in (L^{(p/\theta)'}_{w'},L^{\infty}_{v'})_t(\mathbb R^n)$, by Lemma \ref{DuiO},
we write
\begin{align*}
\lambda\|\chi_{\{x\in\mathbb R^n:|I_\alpha f(x)|>\lambda\}}\|_{(L^p_w, L^q_v)_t(\mathbb R^n)}
    &=\lambda\left\||\chi_{\{x\in\mathbb R^n:|I_\alpha f(x)|>\lambda\}}|^\theta
\right\|^\frac{1}{\theta}_{(L^{p/\theta}_w, L^{q/\theta}_v)_t(\mathbb R^n)}\\
    &=\lambda\left(\int_{\mathbb R^n}\chi_{\{x\in\mathbb R^n:|I_\alpha f(x)|>\lambda\}}(x)
g(x)dx\right)^\frac{1}{\theta}.
\end{align*}
And letting $w:=[M_\eta(|g|)(x)]^\frac{1}{\theta}$ with $0<\eta<1$, we have $w^\theta\in A_1$
and hence $w^\theta\in A_{\theta\frac{n-\alpha}{n}}$.
Then $w\in A_{1,\theta}$. By Lemma \ref{Holder} and Lemma \ref{ProIf}, it can obtain that
\begin{align*}
\lambda\|\chi_{\{x\in\mathbb R^n:|I_\alpha f(x)|>\lambda\}}\|_{(L^p_w, L^q_v)_t(\mathbb R^n)}
    &=\lambda\left[\int_{\mathbb R^n}\chi_{\{x\in\mathbb R^n:|I_\alpha f(x)|>\lambda\}}(x)
M_\eta g(x)dx\right]^\frac{1}{\theta}\\
    &\leq C\int_{\mathbb R^n}|f(x)|w(x)dx\\
    &\leq C\|f\|_{(L^{p_0}_w, L^1_v)_t(\mathbb R^n)}
\|[M_\eta(g)]^\frac1\theta\|_{(L^{p'_0}_{w'},L^{\infty}_{v'})_t(\mathbb R^n)}.
\end{align*}
From Theorem \ref{ThMf}, it follows that
\begin{align*}
\|[M_\eta(g)]^\frac1\theta\|_{(L^{p'_0}_{w'},L^{\infty}_{v'})_t(\mathbb R^n)}
    &=\|[M(|g|^\eta)]^\frac1{\eta\theta}\|
_{(L^{p'_0}_{w'},L^{\infty}_{v'})_t(\mathbb R^n)}=\|M(|g|^\eta)\|_{(L^{{p'_0}/{\eta\theta}}_{w'},
L^{\infty}_{v'})_t(\mathbb R^n)}^\frac 1{\eta\theta}\\
    &\leq C\||g|^\eta\|_{(L^{{p'_0}/{\eta\theta}}_{w'},
L^{\infty}_{v'})_t(\mathbb R^n)}^\frac 1{\eta\theta}=C\|g\|_{(L^{{p'_0}/{\theta}}_{w'},
L^{\infty}_{v'})_t(\mathbb R^n)}^\frac 1\theta.
\end{align*}
since $
\frac 1{p_0'/\theta}=(1-\frac1{p_0})\theta
=(1-\frac1p-\frac{\alpha}n)\theta
=(\frac1\theta-\frac1p)\theta
=1-\frac1{p/\theta}=\frac1{(p/\theta)'}$,
then we have
$$
\|I_\alpha f\|_{W(L^p_w, L^q_v)_t(\mathbb R^n)}\leq C\|f\|_{(L^{p_0}_w, L^1_v)_t(\mathbb R^n)}.
$$
Thus, the result holds.
\end{proof}

\section{Further Remarks}\label{s4}
Notice that the proof of Theorem \ref{ThTf} only needs to use the boundedness of $M$
on $(L^p_w, L^q_v)_t(\mathbb R^n)$ and the boundedness of any operator on $L^p_w(\mathbb R^n)$.
So we introduce some operators, their simple proofs use Theorem \ref{ThMf}
and the boundedness of the operator on $L^p_w(\mathbb R^n)$ and is omitted.
\par Let $\mathbb S^{n-1}(n\ge2)$  be the unit sphere in $\mathbb R^n$ equipped with the normalized
Lebesgue measure $d\sigma$, $\Omega(x)$ is homogeneous of
degree zero on $\mathbb R^n$ and $\Omega\in L^\theta(\mathbb S^{n-1})$
with $1<\theta\leq\infty$ and such that
\begin{equation}\label{Proome}
\int_{\mathbb S^{n-1}}\Omega(x')d\sigma(x')=0
\end{equation}
where
$x'=\frac x{|x|}$ for any $x\neq0$, the homogeneous singular integral operator $T_\Omega$ can be defined by
$$
T_\Omega f(x)=p.v.\int_{\mathbb R^n} \frac{\Omega(y')}{|y|^n}f(x-y)dy,
$$
and the Marcinkiewicz integral of higher dimension $\mu_\Omega$ by
$$
\mu_\Omega f(x)=\lt(\int_{0}^{\infty}\lt|\int_{|x-y|\leq t}
\frac{\Omega(x-y)}{|x-y|^{n-1}}f(y)dy\rt|^2\frac{dt}{t^3}\rt)^\frac12.
$$
\begin{lemma}\label{ProTo}\cite{refTo}
For $\Omega\in L^\theta(\mathbb S^{n-1})$ and $1<\theta<\infty$, if $\theta'\leq p<\infty$
and $w\in A_{p/\theta'}$, then there exist $C>0$ such that
$$
\|T_\Omega (f)\|_{L^p_w(\mathbb R^n)}\le C\|f\|_{L^p_w(\mathbb R^n)}.
$$
If $p=1$, $w\in A_{1}$, then there exist a constant $C>0$ such that
$$
\|T_\Omega(f)\|_{L^{1,\infty}_w(\mathbb R^n)}\le C\|f\|_{L^1_w(\mathbb R^n)}.
$$
\end{lemma}
\begin{theorem}\label{ThTo}
Let $0<t<\infty$, $\Omega\in L^\theta(\mathbb S^{n-1})$, $1<\theta,\gamma<\infty$, $\theta'\leq p<\infty$,
$w\in A_{p/\theta'}$.\\
(a). If $\gamma'\leq q<\infty$, and $v\in A_{q/\gamma'}$, then
$$
||T_\Omega(f)||_{(L_w^p,L_v^q)_t(\mathbb R^n)}\leq C ||f||_{(L_w^p,L_v^q)_t(\mathbb R^n)}.
$$
(b). If $q=1$ and $v\in A_1$, then
$$
||T_\Omega(f)||_{W(L_w^p,L_v^1)_t(\mathbb R^n)}\leq  C||f||_{(L_w^p,L_v^1)_t(\mathbb R^n)}.
$$
The positive constant $C$ is independent of $f$ and $t$.
\end{theorem}
\begin{lemma}\label{Promu}\cite{refmu}
For $\Omega\in L^\theta(\mathbb S^{n-1})$ and $1<\theta\leq\infty$, if $p$, $\theta$, $w$
satisfy one of the following conditions:
\begin{enumerate}[(i).]
         \item $\theta'<p<\infty$, and $w\in A_{p/\theta'}$,
         \item $1<p<\theta$, and $w^{1-\theta'}\in A_{p'/\theta'}$,
         \item $1<p<\infty$, and $w\in A_{p}$.
\end{enumerate}
Then there exist constant $C>0$ such that
$$
\|\mu_\Omega (f)\|_{L^p_w(\mathbb R^n)}\leq C\|f\|_{L^p_w(\mathbb R^n)}.
$$
If $p=1$, $w\in A_{1}$, then there exist a positive constant $C$ such that
$$
\|\mu_\Omega (f)\|_{L^{1,\infty}_w(\mathbb R^n)}\leq C\|f\|_{L^1_w(\mathbb R^n)}.
$$
\end{lemma}
\begin{theorem}\label{Thmu}
Let $0<t<\infty$, $\Omega\in L^\theta(\mathbb S^{n-1})$, $1<p<\infty$ and $w\in A_p$. \\
(a). If $1<q<\infty$ and $v\in A_q$, then there exist constant $C>0$ such that
$$
||\mu_\Omega(f)||_{(L_w^p,L_v^q)_t(\mathbb R^n)}\leq C ||f||_{(L_w^p,L_v^q)_t(\mathbb R^n)}.
$$
(b). If $q=1$ and $v\in A_1$, then there exist constant $C>0$ such that
$$
||\mu_\Omega(f)||_{W(L_w^p,L_v^1)_t(\mathbb R^n)}\leq  C||f||_{(L_w^p,L_v^1)_t(\mathbb R^n)}.
$$
The universal positive constant $C$ is independent of $f$ and $t$.
\end{theorem}

We also define the Bochner-Riesz operator of order $\delta>0$ in terms of Fourier transform by
$$
(T^\delta_R f)^\wedge(\xi)=\lt(1-\frac{|\xi|^2}{R^2}\rt)^\delta_{+}\hat{f}(\xi),
$$
where $\hat{f}$ denote the Fourier transform of $f$. These operators can be defined by
$$
T^\delta_R f(x)=(f*\phi_{1/R})(x),
$$
where $\phi(x)=[(1-|\cdot|^2)^\delta_{+}]^\vee(x)$, and $f^\vee$ is the inverse Fourier transform of
$f$.
\par The associate maximal operator is defined by
$$
T^\delta_*f(x)=\underset{R>0}{\sup}|T^\delta_Rf(x)|.
$$
\begin{lemma}\label{ProT*}\cite{refT*1,refT*2,refT*3}
Let $n\ge 2$. If $1<p<\infty$ and $w\in A_p$, then there exist constant $C$ such that
$$
\lt(\int_{\mathbb R^n}|T^{(n-1)/2}_*f(x)|^qw(x)dx\rt)^\frac{1}{q}
\leq C\lt(\int_{\mathbb R^n}|f(x)|^qw(x)dx\rt)^\frac{1}{q}.
$$
For any $\alpha,R>0$, if $p=1$, $w\in A_{1}$,
then there exist a constant $C$ such that
$$
w(\{x\in \mathbb R^n:T^{(n-1)/2}_R(f)(x)>\alpha\})\le \frac{C}{\alpha}\int_{\mathbb R^n}|f(x)|w(x)dx.
$$
The positive constant $C$ is independent of $f$ and $\alpha$.
\end{lemma}
\begin{theorem}\label{ThT*}
Suppose that $0<t<\infty$, $1<p<\infty$, $w\in {{A}_{p}}$.\\
(a). If $1<q<\infty$, and $v\in A_q$, then we have
$$
||T^{(n-1)/2}_*(f)||_{(L_w^p,L_v^q)_t(\mathbb R^n)}\leq C ||f||_{(L_w^p,L_v^q)_t(\mathbb R^n)}.
$$
(b). For any $R>0$, if $q=1$ and $v\in A_1$, then we have
$$
||T^{(n-1)/2}_R(f)||_{W(L_w^p,L_v^1)_t(\mathbb R^n)}\leq  C||f||_{(L_w^p,L_v^1)_t(\mathbb R^n)}.
$$
The universal constant $C>0$ is independent of $f$ and $t$.
\end{theorem}

Suppose that $\varphi(x)\in L^1(\mathbb R^n)$ satisfies
\begin{equation}\label{Eqphi}
\int_{\mathbb R^n}\varphi(x)dx=0.
\end{equation}
The generalized Littlewood-Paley $g$ function $g_\varphi$ is defined by
$$
g_\varphi(f)=\lt(\int_{0}^{\infty}|\varphi_t*f(x)|^2\frac{dt}{t}\rt)^\frac12,
$$
where $\varphi_t(x)=\frac{1}{t^n}\varphi(\frac xt)$.

\begin{lemma}\label{Prog}\cite{refgf}
Suppose that $\varphi\in L^1(\mathbb R^n)$ satisfies (\ref{Eqphi}) and the following condition
\begin{equation}\label{phi2}
|\varphi(x)|\leq\frac C{(1+|x|)^{n+1}}
\end{equation}
and
\begin{equation}\label{phi3}
|\nabla\varphi(x)|\leq\frac C{(1+|x|)^{n+2}}.
\end{equation}
If $1<p<\infty$, $w\in A_p$, then there exist constant $C>0$ such that
$$
\|g_\varphi(f)\|_{L^p_w(\mathbb R^n)}\leq C\|f\|_{L^p_w(\mathbb R^n)}.
$$
If $p=1$ and $w\in A_1$, then there exist constant $C>0$ such that
$$
\|g_\varphi(f)\|_{L^{1,\infty}_w(\mathbb R^n)}\leq C\|f\|_{L^1_w(\mathbb R^n)}.
$$
\end{lemma}
\begin{theorem}\label{Thg}
Suppose that $\varphi(x)$ satisfy (\ref{Eqphi}),(\ref{phi2}) and (\ref{phi3}),
for $1<p<\infty$ and $w\in A_p$. \\
(a). If $1<q<\infty$ and $v\in A_q$,
then there exist constant $C>0$ such that
$$
\|g_\varphi(f)\|_{(L_w^p,L_v^q)_t(\mathbb R^n)}\leq C ||f||_{(L_w^p,L_v^q)_t(\mathbb R^n)}.
$$
(b). If $q=1$ and $v\in A_1$, then there exist constant $C>0$, such that
$$
\|g_\varphi(f)\|_{W(L_w^p,L_v^1)_t(\mathbb R^n)}\leq C ||f||_{(L_w^p,L_v^1)_t(\mathbb R^n)}.
$$
The positive constant $C$ is independent of $f$ and $t$.
\end{theorem}
The intrinsic square functions were first introduced by Wilson\cite{refSf1,refSf2} which are defined as follows.
For $0<\alpha\leq1$, let $\mathscr C_\alpha$ be the family of functions $\varphi$ defined on
$\mathbb R^n$ such that $\varphi$ has support containing in $\{x\in\mathbb R^n :|x|\leq 1\}$
and satisfy (\ref{Eqphi}), for all $x, y\in\mathbb R^n$,
$$
|\varphi(x)-\varphi(y)|\leq |x-y|^\alpha.
$$
And for $(y,t)\in\mathbb R^{n+1}_{+}=\mathbb R^n\times(0,\infty)$, and $f\in L^1_{loc}(\mathbb R^n)$.
Let
$$
A_\alpha(f)(y,t)=\sup_{\varphi\in\mathscr C_\alpha}|f*\varphi_t(y)|
=\sup_{\varphi\in\mathscr C_\alpha}\lt|\int_{\mathbb R^n} \varphi_t(y-z)f(z)dz\rt|.
$$
Then we define the intrinsic square function of $f$ (of order $\alpha$) by
$$
\mathcal S_\alpha(f)(x)=\lt(\int\int_{\Gamma(x)}(A_\alpha(f)(y,t))^2\frac{dydt}{t^{n+1}}\rt)^\frac12,
$$
where
$
\Gamma(x)=\lt\{(y,t)\in\mathbb R^{n+1}_{+}:|x-y|<t\rt\}
$
and $\varphi_t(x)=\frac{1}{t^n}\varphi(\frac xt)$.
\begin{lemma}\label{ProSf}\cite{refSf1}
Let $0<\alpha\leq 1$, if $1<p<\infty$, and $w\in A_p$, then there exist constant $C>0$ such that
$$
\|\mathcal S_\alpha(f)\|_{L^p_w(\mathbb R^n)}\leq C\|f\|_{L^p_w(\mathbb R^n)}.
$$
If $p=1$, $w\in A_1$, then there exist $C>0$ such that
$$
\|\mathcal S_\alpha(f)\|_{L^{1,\infty}_w(\mathbb R^n)}\leq C\|f\|_{L^1_w(\mathbb R^n)}.
$$
\end{lemma}

\begin{theorem}\label{ThSf}
Let $0<\alpha\leq 1$. For $1<p<\infty$ and $w\in A_p$.\\
(a). If $1<q<\infty$ and $v\in A_q$,
then there exist constant $C>0$, such that
$$
\|\mathcal S_\alpha(f)\|_{(L_w^p,L_v^q)_t(\mathbb R^n)}\leq C ||f||_{(L_w^p,L_v^q)_t(\mathbb R^n)}.
$$
(b). If $q=1$ and $v\in A_1$, then there exist constant $C>0$, such that
$$
\|\mathcal S_\alpha(f)\|_{W(L_w^p,L_v^1)_t(\mathbb R^n)}\leq C ||f||_{(L_w^p,L_v^1)_t(\mathbb R^n)}.
$$
The positive constant $C$ is independent of $f$ and $t$.
\end{theorem}
\hspace*{-0.6cm}\textbf{\bf Competing interests}\\
The authors declare that they have no competing interests.\\

\hspace*{-0.6cm}\textbf{\bf Funding}\\
The research was supported by Natural Science Foundation of China (Grant No. 12061069).\\

\hspace*{-0.6cm}\textbf{\bf Authors contributions}\\
All authors contributed equality and significantly in writing this paper. All authors read and approved the final manuscript.\\

\hspace*{-0.6cm}\textbf{\bf Acknowledgments}\\
All authors would like to express their thanks to the referees for valuable advice regarding previous version of this paper.\\

\hspace*{-0.6cm}\textbf{\bf Authors detaials}\\
Yuan Lu(Luyuan\_y@163.com) and Jiang Zhou(zhoujiang@xju.edu.cn),
College of Mathematics and System Science, Xinjiang University, Urumqi, 830046, P.R China.\\
Songbai Wang(haiyansongbai@163.com),
College of Mathematics and Statistics, Chongqing Three Gorges University, Chongqing 404130, P.R China.\\
	

\begin{thebibliography}{99}
\frenchspacing

\bibitem{APA}
Auscher, P., Prisuelos-Arribas, C.,
Tent space boundedness via extrapolation,
Mathematische Zeitschrift, 286:  1575-1604 (2017).

\bibitem{AM}
Auscher, P., Mourgoglou, M.,
Representation and uniqueness for boundary value elliptic problems via first order systems,
Rev. Mat. Iberoam, 35: 241-315 (2019).

\bibitem{refJCC}
Bertrandias, J., Datry, C. Dupuis, C.,
Unions et intersections d.espaces $L^p$ invariantes par translation ou convolution,
Annales Institut Fourier, 28(2):53-84 (1978).

\bibitem{5} Cowling, M., Meda, S., Pasquale, R.,
Riesz potentials and amalgams,
Annales Institut Fourier, 49(4):1345-1367 (1999).

\bibitem{refwcha} Cruz-Uribe, B., Jos\'eMariaMartell and CarlosP\'erez,
Weights, extrapolation and the theory of Rubio de Francia,
Springer Basel, (2011).

\bibitem{refTo}Duoandikoetxea, J.,
Weighted norm inequalities for homogeneous singular integrals,
Trans. Am. Math. Soc, 336: 869-880 (1993).


\bibitem{8}Duoandikoetxea, J.,
Fourier Analysis,
Grad. Studies in Math. 29, American Mathematical Society, Providence, RI (2001).

\bibitem{4} Fournier, J.,
On the Hausdorff-Young theorem for amalgams,
Monatshefte Fur Mathematik, 95(2):117-135 (1983).

\bibitem{CMFA}Grafakos, L.,
Classical and Modern Fourier Analysis,
Pearson Education, Upper Saddle River, NJ, USA (2004).

\bibitem{Duiou} Heil, C.,
Wiener amalgam spaces in generalized harmonic analysis and wavelet theory,
Ph.D.thesis, University of Maryland, CollegePark, MD (1990).

\bibitem{2}
Holland, F.,
Harmonic analysis on amalgams of $L^p$ and $l^q$,
Journal of the London Mathematical Society, 10(2): 295-305 (1975).


\bibitem{refT*3}H. Wang,
Some estimates for Bochner-Riesz operators on the weighted Morrey spaces,
Acta Math. Sinica (Chin. Ser.), 55(3): 551-560 (2012).

\bibitem{6} Kellogg, C. N.,
An extension of the Hausdorff-Young theorem,
 Michigan Mathematical Journal, 18(2):121-127 (1971).

\bibitem{Ho}
K-P Ho,
Operators on Orlicz-slice spaces and Orlicz-slice Hardy spaces,
J. Math. Anal. Appl., 503, 125279 (2021).

\bibitem{MP}Martell, J. M. , Prisuelos-Arribas, C.,
Weighted Hardy spaces associated with elliptic operators Part: I. Weighted norm inequalities
for conical square functions,
Trans Amer Math Soc, 369(6): 4193-4233 (2017).

\bibitem{15}Muckenhoupt, B.,
Weighted norm inequalities for the Hardy maximal function,
Trans. Am. Math. Soc, 165, 207-226 (1972).

\bibitem{16} Muckenhoupt, B., Wheeden, R.L,
Weighted norm inequalities for fractional integrals,
Trans. Am. Math.Soc, 192: 261-274 (1974).


\bibitem{refgf}S. Lu, Y. Ding, D. Yan,
Singular Integrals and Related Topics,
World Scientific Publishing Co. Pte. Ltd., Hackensack, NJ (2007).

\bibitem{W} S. Wang,
Generalized Orlicz-slice spaces,
extrapolation and applications, submitted.

\bibitem{7} Szeptycki P.,
Some remarks on the extended domain of Fourier transform,
Bulletin of the American Mathematical Society, 73:398-402 (1967).

\bibitem{refT*2}Vargas, A.,
Weighted weak type (1,1) bounds for rough operators,
J. London Math. Soc., 54: 297-310 (1996).

\bibitem{refSf1}Wilson, M.,
Weighted Littlewood-Paley Theory and Exponential-Square Integrability,
vol. 1924 of Lecture Notes in Math, Springer (2007).

\bibitem{refSf2}Wilson, M.,
The intrinsic square function,
Revista Matematica Iberoamericana,
23(3):771-791 (2007).

\bibitem{1}
Wiener, N.,
On the representation of functions by trigonometrical integrals,
Mathematische Zeitschrift, 24(1): 575-16 (1926).

\bibitem{refT*1} X. L. Shi, Q. Y. Sun,
Weighted norm inequalities for Bochner-Riesz operators and singular integral operators,
Proc. Am. Math. Soc, 116: 665-673 (1992).

\bibitem{refmu}Y. Ding, D. S. Fan, Y. B. Pan,
Weighted boundedness for a class of rough Marcinkiewicz integrals,
Indiana Univ. Math., 48: 1037-1055 (1999).

\bibitem{ZYYW}
Y. Zhang, D. Yang, W. Yuan, S. Wang,
Real-Variable Characterizations of Orlicz-Slice Hardy Spaces,
Anal. Appl, 17: 597-664 (2019).


\end{thebibliography}
\end{document}